\documentclass[12pt]{amsart}

\usepackage{amsmath, amsthm, amsfonts, amssymb, latexsym, mathrsfs,
  stmaryrd}
\usepackage{hyperref}
\usepackage{enumitem}
\usepackage{graphicx}
\usepackage{booktabs,colortbl}

\renewcommand{\today}{%
  \number\day\space
  \ifcase\month\or
  January\or February\or March\or April\or May\or June\or
  July\or August\or September\or October\or November\or December\fi
  \space \number\year}

\newtheorem{theorem}{Theorem}[section]

\newtheorem{corollary}[theorem]{Corollary}
\newtheorem{lemma}[theorem]{Lemma}
\newtheorem{fact}[theorem]{Fact}

\theoremstyle{definition}
\newtheorem{defn}[theorem]{Definition}
\newtheorem{notation}[theorem]{Notation}

\newtheorem{example}[theorem]{Example}

\theoremstyle{remark}

\makeatletter

\newenvironment{dotto}{\buildrel\textstyle.\over{\hbox{\vrule height3pt depth0pt width0pt}{\smash\to}}}

\newenvironment{ds}{\displaystyle}

\newenvironment{g}{G\"{o}del~}

\newenvironment{lo}{{\L}ukasiewicz~}

\makeatother

\title[Compactness of first-order fuzzy logics]{Compactness of first-order fuzzy logics}

\author[S. M. A. Khatami]{Seyed Mohammad Amin Khatami}
\address{Seyed Mohammad Amin Khatami\\
  Department of Computer Science\\
  Birjand University of Technology\\
  Birjand\\
  Iran  }
\email{http://birjandut.ac.ir}

\thanks{}
\thanks{}

\begin{document}
\begin{abstract}
One of the nice properties of the first-order logic is the compactness of satisfiability. It state that a finitely satisfiable theory is satisfiable. However, different degrees of satisfiability in many-valued logics, poses various kind of the compactness in these logics. One of this issues is the compactness of $K$-satisfiability.

Here, after an overview on the results around the compactness of satisfiability and compactness of $K$-satisfiability in many-valued logic based on continuous t-norms (basic logic), we extend the results around this topic. To this end, we consider a reverse semantical meaning for basic logic. Then we introduce a topology on $[0,1]$ and $[0,1]^2$ that the interpretation of all logical connectives are continuous with respect to these topologies. Finally using this fact we extend the results around the compactness of satisfiability in basic ogic.
\end{abstract}
\maketitle                   





\section{Introduction}
The compactness theorem in classical first-order logic state that a finitely satisfiable theory is satisfiable. In the case of many-valued logics, switching from bivalent of the truth value set to many-valent, poses different kinds of many valued logics as well as various kinds of the compactness in these logics. The truth value set, basic set of logical connectives, interpretations of logical connectives, and different kinds of satisfiability, are the most significant factors that impact on the logic. The class of all many valued logics is very large to study. However, as the metamathematics of continuous t-norm based many valued logics have been studied in \cite{hajek98}, we shall study the compactness in these logics. Remind that a continuous t-norm $T$ is a continuous function $T:[0,1]^2\to [0,1]$ ($[0,1]$ and $[0,1]^2$ with the Euclidean topology) which is commutative, associative, non-decreasing on both arguments, and $T(1,x)=x$ for all $x\in[0,1]$. The main examples of continuous t-norms are: \lo, \g, and product t-norm. It is well-known that each continuous t-norm is a combination of these three fundamental continuous t-norms (see e.g. \cite{hajek98}).

For propositional fuzzy logics based on continuous t-norms, a systematic study have been done for the usual compactness as well as the K-compactness in \cite{cintnav2004}. However, in the case of predicate fuzzy logics, there is no such a comprehensive account. In many cases, in fact, even the usual compactness fails in these logics. Examples \ref{godelf} and \ref{prodf} shows that the usual compactness fails in the \g and product logic whose set of truth values is the continuous scale $[0,1]$. In spite of these examples, however, changing the truth value set or generalizing the concept of satisfiability to K-satisfiability, leads to some version of the compactness in these logics.

One of the fuzzy logics that satisfies the usual compactness as well as the K-compactness
for any closed subset K of the unite interval $[0,1]$,
in both propositional and first-order cases, is the \lo logic \cite{cintnav2004, but1995, tpd2011}. In fact, the main reson behind this, is the continuity of truth function of logical connectives of the \lo logic with respect to the Euclidean topology on $[0,1]$. In the case of propositional \lo logic an easy application of the Tychonoff theorem leads to the result \cite{cintnav2004, but1995}. In first-order case, there are several methods, of which the most significant one is the "Ultraproduct method" \cite{tpd2011,del2014}.

Here we extends the ideas in \cite{tpd2011} and \cite{kp2015} to solve the open problem stated in \cite{del2014} about a systematic study around the compactness and K-compactness of first-order fuzzy logics.
As mentioned, the main reason that ultraproduct method works well for the \lo
logic is the continuity of the truth function of connectives with respect to the Euclidean topology
on the standard truth value set $[0,1]$. On the other hand, one can easily verify that the truth function of
$\neg(p\leftrightarrow q)$ in \lo logic is the Euclidean metric $d(x,y)=|x-y|$, while in \g logic or product logic this gives only the discrete metric.

If one consider a reverse semantical meaning on the set of truth values $[0,1]$, i.e., if $0$ is stands for absolute
truth and $1$ for absolute falsity, then the truth function of the equivalence connective in \lo logic becomes
the Euclidean metric $d(x,y)=|x-y|$, and also in \g logic it's truth function is the  metric $d_G:[0,1]^2\to [0,1]$ defined by
\begin{center}
$d_G(x,y)=\left\{\begin{array}{ll}
\max\{x,y\}&x\ne y\\
0&x=y
\end{array}\right.$
\end{center}
and in product logic it's truth function is the metric $d_\pi:[0,1]^2\to [0,1]$ defined by
\begin{center}
$d_\pi(x,y)=\left\{\begin{array}{ll}
\displaystyle\frac{|x-y|}{1-\min\{x,y\}}&x\ne y\\
0&x=y
\end{array}\right.$.
\end{center}

Considering this fact, we prove some versions of the compactness for \g logic and
product logic by the ultraproduct method and then extend the result to fuzzy logics based on
continuous t-norms. As the first step after introduction, we have a review on some facts about the three fundamental continuous t-norm based fuzzy logics (\lo, \g, and product logic). Section 3 presents a reverse semantical meaning of fuzzy logics, and then we prove some variant of the compactness for these three basic fuzzy logics. Finally, we translate results to every-day semantic of fuzzy logics.
\section{Propositional Basic Logic}\label{sec1}
Continuous t-norm based fuzzy logics may be presented as having the truth value set $[0,1]$ with its natural ordering in which $1$ standing for absolute truth and $0$ for absolute falsity. Basic logical connectives are $\{\&,\to,\bot\}$.
\begin{defn}\label{bllogic}
Let $P=\{p_i\}_{i\in I}$ be a set of atomic propositions. Assume that $Prop$ be
generated from $P$ by the formal binary operations $\{\&, \to\}$ and the unary operation $\bot$. $Prop$ is called
a propositional basic logic and denoted by {\bf BL}.
\end{defn}
The strong conjunction $\&$ is interpreted by a continuous t-norm $T$, implication is interpreted by residuum of $T$ which is defined by $x\Rightarrow_Ty=\sup\{z: T(z,x)\le y\}$, and the zero function plays the role of $\bot$.

Among well known continuous t-norms based fuzzy logics, one can mention to the \lo, \g, and product logic whose corresponding t-norms and residua are listed in Table~\ref{tnorms}.
\begin{table}[h]
\centering
\caption{Some continuous t-norms and their residua}
\label{tnorms}
\begin{tabular}{l l l l}
\toprule
\rowcolor[gray]{.8}
logic name&t-norm&residuum\\
\midrule
\lo&$T_L(x,y)=\max\{0,x+y-1\}$&$x\Rightarrow_L y=
\left\{\begin{array}{cc}
1&x\le y\\
1-x+y&x>y
\end{array}\right.$\\
\g&$T_G(x,y)=\min\{x,y\}$&$x\Rightarrow_G y=
\left\{\begin{array}{cc}
1&x\le y\\
y&x>y
\end{array}\right.$\\
Product&$T_\pi(x,y)=x.y$&$x\Rightarrow_\pi y=
\left\{\begin{array}{cc}
1&x\le y\\
y/x&x>y
\end{array}\right.$\\
\bottomrule
\end{tabular}
\end{table}
\begin{defn}
Any function $v_0:P\to [0,1]$ could be extended to a unique function $v$, called an evaluation, from the set of all propositions to $[0,1]$ by the following rules:
\begin{center}
$v(\bot)=0$, $v(\varphi \&\psi)=T\left(v(\varphi),v(\psi)\right)$, and $v(\varphi\to \psi)=v(\varphi)\Rightarrow v(\psi)$.
\end{center}
\end{defn}
Other connectives that are commonly used in {\bf BL} are defined in Notation \ref{connectives}.
\begin{notation}\label{connectives} Further logical connectives that are defined by the set of basic logical connectives are:
\begin{itemize}
\item[] $\varphi\wedge\psi:=\varphi\&(\varphi\to\psi)$
\item[] $\varphi\vee\psi:=\big((\varphi\to\psi)\to\psi\big)\wedge\big((\psi\to\varphi)\to\varphi\big)$
\item[] $\neg\varphi:=\varphi\to\bot$
\item[] $\varphi\leftrightarrow\psi:=(\varphi\to\psi)\&(\psi\to\varphi)$
\item[] $\top:=\neg\bot$
\end{itemize}
\end{notation}
Using the continuity of t-norm, one can easily verify that $v(\varphi\vee\psi)=\max\{v(\varphi),v(\psi)\}$ and
$v(\varphi\wedge\psi)=\min\{v(\varphi),v(\psi)\}$.
\begin{defn}
Let $v$ be an evaluation and $\Sigma\cup\{\varphi\}\subseteq Prop$.
If $v(\varphi)=1$ we say that $v$ models $\varphi$, in symbols $v\models\varphi$. $v$ models $\Sigma$, $v\models\Sigma$, whenever $v\models\psi$ for all $\psi\in\Sigma$. When a proposition or theory has a model we call it satisfiable. We say that $\Sigma$ entails $\varphi$ whenever all models of $\Sigma$ models $\varphi$ which is denoted by $\Sigma\models\varphi$.
\end{defn}
As the set of truth values assumed to be $[0,1]$ instead of the finite two valued set $\{0,1\}$, the concept of satisfiability is, to some extent, a crisp notion. One of the generalization of this concept to a fuzzy concept, is   $K$-satisfiability.
\begin{defn}
For $K\subseteq[0,1]$ a proposition $\varphi$ is said to be $K$-satisfiable if there exists an evaluation $v$ such that $v(\varphi)\in K$. In this way $v$ is called a $K$-model of $\varphi$. A theory whose propositions satisfied by a $K$-model $v$, is called a $K$-satisfiable theory.
\end{defn}
While {\bf BL} is the logic of all continuous t-norms,  the known weakest many-valued logic based on t-norms is the logic of left-continuous t-norms, {\bf MTL}, whose basic logical connectives are $\{\&,\to,\wedge,\bot\}$ which are interpreted respectively by a left continuous t-norm, its residua, minimum and falsum.
\section{Axioms}
As H{\'a}jek mentioned, the axioms of {\bf BL} are as the following statements.\cite{hajek98},
\begin{itemize}
  \item[(A1~~)] $(\varphi \rightarrow \psi)\rightarrow \big((\psi\rightarrow \chi)\rightarrow(\varphi\rightarrow
\chi)\big)$
  \item[(A2~~)] $(\varphi\&\psi)\rightarrow \varphi$
  \item[(A3~~)] $(\varphi\&\psi)\rightarrow(\psi\&\varphi)$
  \item[(A4~~)] $\big(\varphi\&(\varphi\to\psi)\big)\rightarrow\big(\psi\&(\psi\to\varphi)\big)$
  \item[(A5a)] $\big(\varphi\rightarrow(\psi\rightarrow\chi)\big)\rightarrow\big((\varphi\&\psi)\rightarrow
\chi\big)$
  \item[(A5b)] $\big((\varphi\&\psi)\rightarrow
\chi\big)\rightarrow\big(\varphi\rightarrow(\psi\rightarrow\chi)\big)$
  \item[(A6~~)] $\big((\varphi\rightarrow\psi)\rightarrow\chi\big)\rightarrow
      \Big(\big((\psi\rightarrow\varphi)\rightarrow\chi\big)\rightarrow\chi\Big)$
  \item[(A7~~)] $\bot\rightarrow\varphi$
\end{itemize}
The only inference rule is being modus ponens. The concept of proof, which is denoted by $\vdash$, is defined in natural way.

{\bf BL} proves many interesting properties which could be find in literature. The following Lemma includes those that we need here.
\begin{lemma}
{\bf BL} proves the following properties.
\begin{enumerate}[label={\thelemma.\arabic*}]
  \item $\varphi\to(\psi\to\varphi)$ \label{l1}
  \item $(\varphi\&\psi)\to(\varphi\wedge\psi)$ \label{l2}
  \item $\big((\varphi_1\to\psi_1)\&(\varphi_2\to\psi_2)\big)\to
  \big((\varphi_1\&\varphi_2)\to(\psi_1\&\psi_2)\big)$ \label{l3}
\end{enumerate}
\end{lemma}
\begin{proof}
See \cite{hajek98}.
\end{proof}
\section{First-Order Basic Logic}
Given a first order language $\mathcal{L}$ consist of function symbols $\{f_i\}_{i\in I}$ and predicate symbols $\{P_j\}_{j\in J}$, the concept of $\mathcal{L}$-structure is defined as usual.
\begin{defn}
An $\mathcal{L}$-structure $\mathcal{M}$ is a nonempty set $M$ together with a set of functions
$\{f_i^{\mathcal{M}}:M^{n_i}\to M\}_{i\in I}\cup\{P_j^{\mathcal{M}}:M^{n_j}\to [0,1]\}_{j\in J}$ as the interpretations of language symbols, assuming that whenever $n_i=0$, $f_i^{\mathcal{M}}$ is an element of $M$ and whenever $n_j=0$, $P_j^{\mathcal{M}}$ is a truth value in $[0,1]$. Note that nullary function symbols
of the language $\mathcal{L}$ are commonly called constant symbols and denoted by $c_i$ instead of $f_i$.
\end{defn}
\begin{defn}
For an $n$-tuple $\bar{x}=(x_1, ..., x_n)$, the interpretation of term $t(\bar{x})$ is a functions $t^\mathcal{M}:M^n\to M$ defined inductively by 1) if $t(\bar{x})=x_i$ then $t^\mathcal{M}(\bar{a})=a_i$, 2)
if $t(\bar{x})=c$ then $t^\mathcal{M}(\bar{a})=c^{\mathcal{M}}$, and 3)
if $t(\bar{x})=f(t_1(\bar{x}),...,t_n(\bar{x}))$ then $t^\mathcal{M}(\bar{a}) = f^\mathcal{M}(t_1^\mathcal{M}(\bar{a}),...,t_n^\mathcal{M}(\bar{a}))$.\\
Also the interpretation of formula $\varphi(\bar{x})$ is a function $\varphi^\mathcal{M}:M^n \to [0,1]$,
defined inductively as follows:
\begin{itemize}
\item $\bot^\mathcal{M}=0$.
\item For every n-ary predicate symbol $P$, $P(t_1,...,t_n)^\mathcal{M}(\bar{a})=P^\mathcal{M}\left(t_1^\mathcal{M}(\bar{a}),...,t_n^\mathcal{M}(\bar{a})\right)$.
\item $(\varphi \&\psi)^\mathcal{M}(\bar{a}) =T\left(\varphi^\mathcal{M}(\bar{a}),\psi^\mathcal{M}(\bar{a})\right)$.
\item $(\varphi \to \psi)^\mathcal{M}(\bar{a}) =\varphi^\mathcal{M}(\bar{a})
\Rightarrow_T\psi^\mathcal{M}(\bar{a})$.
\item For $\varphi(\bar{x})=\forall y\ \psi(y,\bar{x})$ $\varphi^\mathcal{M}(\bar{a})=\displaystyle\inf_{b\in M}\{\psi^\mathcal{M}(b,\bar{a})\}$.
\item For $\varphi(\bar{x})=\exists y\ \psi(y,\bar{x})$ $\varphi^\mathcal{M}(\bar{a})=\displaystyle\sup_{b\in M}\{\psi^\mathcal{M}(b,\bar{a})\}$.
\end{itemize}
\end{defn}
\begin{defn}
For an $\mathcal{L}$-sentence $\varphi$, we say that $\mathcal{M}$ models $\varphi$, or $\mathcal{M}$ satisfies $\varphi$, or $\varphi$ is satisfiable, whenever $\varphi^\mathcal{M}=1$ and we show this by writing $\mathcal{M}\models\varphi$. An $\mathcal{L}$-theory $\Sigma$, i.e a set of $\mathcal{L}$-sentences, is satisfiable, whenever all of its sentences are satisfied by an $\mathcal{L}$-structure $\mathcal{M}$, denoted by $\mathcal{M}\models \Sigma$. We say that a theory $\Sigma$ entails a sentence $\varphi$, in symbols $\Sigma\models\varphi$, when each model of $\Sigma$ models $\varphi$.\\
For a set $K\subseteq[0,1]$, an $\mathcal{L}$-sentence $\varphi$ is called $K$-satisfiable if
there exists an $\mathcal{L}$-structure $\mathcal{M}$ such that $\varphi^\mathcal{M}\in K$, and $\mathcal{M}$ is
called a $K$-model of $\varphi$. The concept of $K$-satisfiable theory and $K$-entailment, defined in a similar way.
\end{defn}
\section{Compactness and $K$-Compactness in Basic Logic}
As usual a theory $\Sigma$ is finitely satisfiable means that every finite subset of $\Sigma$ is satisfiable. A logic is said to satisfies the compactness property if every finitely satisfiable theory is satisfiable. finitely $K$-satisfiable theory and $K$-compactness defined in a similar way.

Let's remind some known facts about compactness in basic logic.
\subsection{\lo logic}\hfill\\
Let {\L} and {\L}$\forall$ be an abbreviations for the propositional \lo logic and first-order \lo logic.
\begin{fact}\label{kcomplop}
Let $K$ be a compact subset of $[0,1]$ in Euclidean topology. Every finitely $K$-satisfiable theory over {\L} is $K$-satisfiable.
\end{fact}
\begin{fact}\label{nonkcomplo}
Let $K$ be a noncompact subset of $[0,1]$ in Euclidean topology. There is a finitely $K$-satisfiable theory over {\L} such that it is not $K$-satisfiable.
\end{fact}
\begin{fact}\label{kcomplo}
Let $K$ be a compact subset of $[0,1]$ in Euclidean topology. Every finitely $K$-satisfiable theory over {\L}$\forall$ is $K$-satisfiable.
\end{fact}
The main reason behind Fact \ref{kcomplop} - Fact \ref{kcomplo} is the continuity of the interpretation of logical connectives in {\L} and {\L}$\forall$. For $K=\{1\}$, Fact \ref{kcomplop} is the standard compactness and it is an easy consequence of the completeness theorem which has been proved independently in \cite{rose58} and \cite{chang1959new}. For arbitrary compact subset $K$ of $[0,1]$, the sufficiency condition for the $K$-compactness of {\L}, Fact \ref{kcomplop}, has been established in \cite{but1995,cintnav2004} and the necessity condition, Fact \ref{nonkcomplo}, has been appeared in \cite{cintnav2004}. Fact \ref{kcomplo} for $K=\{1\}$, is the standard compactness theorem for {\L}$\forall$ that was initially proved in \cite{chang63}. Fact \ref{kcomplo} actually is the  sufficiency condition for the $K$-compactness of {\L}$\forall$ for arbitrary compact subset $K$ of $[0,1]$, and it is proved in \cite{tpd2011}.
\subsection{\g logic and product logic}\hfill\\
The non-continuity of the interpretation of the implication connective in \g logic ({\bf G}) as well as product logic ($\mathbf{\Pi}$), break down getting a general result about the compactness in these logics. However, some partial results are obtained in literature.
\begin{fact}\label{kcompgodelprop1}
Let $K$ be an arbitrary subset of $[0,1]$ and the set of atomic propositions be finite. Then every finitely $K$-satisfiable theory over the propositional \g logic,{\bf G}, is $K$-satisfiable.
\end{fact}
\begin{fact}\label{comgodelp}
Assume that the set of atomic propositions is at most countable. Then every finitely satisfiable theory over {\bf G} is satisfiable.
\end{fact}
\begin{fact}\label{comgodel}
Assume that $\mathcal{L}$ be an at most countable first-order language.
In the first-order \g logic {\bf G}$\forall$, every finitely satisfiable $\mathcal{L}$-theory is satisfiable.
\end{fact}
\begin{fact}\label{finitek}
Let $K$ be a finite subset of $[0,1]$. {\bf G} with at most countable set of atomic propositions and {\bf G$\forall$} with at most countable underlying language are $K$-compact.
\end{fact}
\begin{fact}\label{pour}
Let $\mathcal{L}$ be an at most countable first-order language and $K$ be a closed subset of $[0,1]$. {\bf G$\forall$} is not $K$-compact if and only if $K$ is infinitely and $1\notin K$.
\end{fact}
\begin{fact}\label{kcompgodelprop}
Assume that $K\subseteq(0,1]$ containing $1$. Then {\bf G} as well as $\mathbf{\Pi}$ is $K$-compact.
\end{fact}
Fact \ref{kcompgodelprop1} is an easy consequence of the semantic of \g logic. Indeed, since the set of atomic propositions is finite, we can only form finitely many formulas with different semantic.
The common idea in the proof of Fact \ref{comgodelp} and Fact \ref{comgodel} is that the \g algebra of $T$-equivalent formulas could be embedded into the standard \g algebra $[0,1]$. It seems that this idea is originated by Dummet \cite{Dummett59} to prove the completeness theorem for {\bf G} which implies Fact \ref{comgodelp} (see also \cite{hajek98}). This idea is also used by Horn \cite{horn69} to prove the completeness theorem for {\bf G}$\forall$ which argues Fact \ref{comgodel} (again, see also \cite{hajek98}). An easy consequence of Facts \ref{comgodelp} and
\ref{comgodel} is Fact \ref{finitek} \cite{cintnav2004}. A more interesting consequence of the Fact \ref{comgodel} is derived by \cite{pourmahtavana2012} which is given in Fact \ref{pour}. Fact \ref{kcompgodelprop} is proved using the interpretation of double negation and the compactness theorem in classical logic \cite{cintnav2004}. Remind that double negation in \g logic and  product logic is interpreted by the following function.
\begin{center}
$\neg\neg x=\left\{\begin{array}{cc}
1&x>0\\
0&x=0
\end{array}\right.$.
\end{center}
Uncountability of the underlying language in Fact \ref{comgodel} leads to the collapse of the compactness in {\bf G}$\forall$.
‎\begin{example}\label{godelf}‎
‎Let $\mathcal{L}$ be a relational‎ ‎language contains uncountably many unary predicate‎
‎symbols $\{R(x)\}\cup\{\rho_i(x)\}_{i\in\omega_2}$‎. ‎Set‎,
‎\begin{center}‎
‎$T=\Big\{\neg\forall x\,R(x)‎, ‎\forall x\,\Big(\big(R(x)\to\rho_1(x)\big)\to R(x)\Big)\Big\}‎
‎\cup\Big\{\forall x\,\Big(\big(\rho_j(x)\to\rho_i(x)\big)\to\rho_j(x)\Big)‎: ‎i>j\Big\}_{i,j\in\omega_2}$‎.
‎\end{center}‎
Remind that in \g logic
\begin{center}
$\neg\varphi^\mathcal{M}(\bar{a})=\left\{\begin{array}{cc}
1&\varphi^\mathcal{M}(\bar{a})=0\\
0&\varphi^\mathcal{M}(\bar{a})>0
\end{array}\right.~~~~~,$
$~~~~~\left((\varphi \to \psi)\to\varphi\right)^\mathcal{M}(\bar{a}) =\left\{\begin{array}{cc}
1&\varphi^\mathcal{M}(\bar{a})<\psi^\mathcal{M}(\bar{a})<1\\
\psi^\mathcal{M}(\bar{a})&\varphi^\mathcal{M}(\bar{a})\ge\psi^\mathcal{M}(\bar{a})
\end{array}\right.$.
\end{center}
Assume that (in \g logic) $\mathcal{M}\models T$. Thus
\begin{itemize}
\item $\mathcal{M}\models\neg\forall x\,R(x)$ and so there‎ ‎is an element $a\in M$ such that $R^\mathcal{M}(a)<1$‎,
\item‎ $\mathcal{M}\models \forall x\,\Big(\big(R(x)\to\rho_1(x)\big)\to R(x)\Big)$‎, ‎thus‎ ‎$\rho_1^\mathcal{M}(a)<R^\mathcal{M}(a)<1$‎,
\item ‎$\mathcal{M}\models\forall x\,\Big(\big(\rho_j(x)\to\rho_i(x)\big)\to\rho_j(x)\Big)$ for every $i>j\in\omega_2$‎. so we have ‎
\begin{center}
$\rho_{\omega_2}^\mathcal{M}(a)<...<\rho_2^\mathcal{M}(a)<\rho_1^\mathcal{M}(a)<R^\mathcal{M}(a)<1$‎.
\end{center}
\end{itemize}
‎a contradiction with the cardinality of $[0,1]$‎. ‎But‎, ‎one can easily verify that $T$ is finitely satisfiable‎.
‎\end{example}‎
In the case of propositional \g logic, however the expressive power of the language prevent us to offer a similar counter example. Indeed we could no express that the truth value of a proposition is strictly less than $1$.
Yet, if $K$ be an infinite subset of $[0,1)$, then the following example show that with an uncountable set of atomic propositions, the $K$-compactness does not hold in {\bf G}.
\begin{example}\label{counterexamplegodelp}
Assume that $K$ be an infinite subset of $[0,1)$ and $T=\{(p_i\to p_j\}_{i\le j, i,j\in \omega_2}$. As K is infinite, every finite subset $T_f$ of $T$ is $K$-satisfiable. Indeed if
\begin{center}
$\ds m=\min\{i: (p_i\to p_j)\in T_f$ for some $j\}$ and $\ds M=\max\{j: (p_i\to p_j)\in T_f$ for some $i\}$,
\end{center}
then we can choose a $K$-evaluation $v$ such that $v(p_m)>...>v(p_M)$, and so $v$ is a $K$-model of $T_f$.
But since the cardinality of $K$ is at most $\omega_1$, $T$ is not satisfiable.
\end{example}
$K$-Compactness fails over $\mathbf{\Pi}$ even for finitely many atomic symbols \cite{cintnav2004}[Theorem 6.2]. The following example show a similar result for $\mathbf{\Pi}\forall$.
‎\begin{example}\label{prodf}‎
‎Let $\mathcal{L}=\{R‎, ‎\rho\}$ be a relational language in which $R$ and $\rho$ are unary predicate symbols. Assume that‎
\begin{center}‎
$T=\Big\{\neg\forall x\,\big(R(x)\vee\rho(x)\big)‎, ‎\neg\neg\forall x\,R(x)‎, ‎\forall x\,\big( R(x)\to\rho^n(x)\big)\Big\}$‎.
\end{center}‎
‎If (in product logic) $\mathcal{M}\models T$‎, ‎then $\mathcal{M}\models\neg\forall x\,\big(R(x)\vee\rho(x)\big)$ and so there‎
‎is an element $b\in M$ such that $\max\{R^\mathcal{M}(b),\rho^\mathcal{M}(b)\}<1$‎. On the other hand‎, ‎
‎$\mathcal{M}\models\neg\neg\forall x\,R(x)$ and so ‎$R^\mathcal{M}(a)>0$ for all $a\in M$, particularly ‎
‎$0<R^\mathcal{M}(b)<1$‎. But,
‎$\mathcal{M}\models\forall x\,\big( R(x)\to\rho^n(x)\big)$‎, for each $n\ge 1$‎, and so we have
‎$\displaystyle\inf_{a\in M}\big(R^\mathcal{M}(a)\to(\rho^n)^\mathcal{M}(a)\big)=1$‎.
Whence $\big(R^\mathcal{M}(b)\to(\rho^n)^\mathcal{M}(b)\big)=1$‎. So‎
‎$0<R^\mathcal{M}(b)\le(\rho^n)^\mathcal{M}(b)<\rho^\mathcal{M}(b)<1$ for all $n\ge 1$‎, that is impossible‎.
‎Thus $T$ is not satisfiable‎. ‎However‎, ‎obviously $T$ is finitely satisfiable‎.
‎\end{example}‎
In the rest of the paper we develop the results about compactness and $K$-compactness for continuous t-norm based fuzzy logics, specially for \g and product logic.
\section{Metrically Semantic for Basic Logic}
The most popular choice of semantic in fuzzy logics based on the truth value set $[0,1]$ in which $1$ is considered for absolute truth and $0$ for absolute falsity. This semantic is not sanctified, however, and we use a reverse semantical meaning fits more for our purpose, that is $0$ and $1$ represents absolute truth and absolute falsity, respectively. Indeed, this semantic makes the interpretation of the equivalence connective a metric that the interpretation of all logical connectives are continuous with respect to it's induced topology on $[0,1]$. Because of this reason, we call this semantic "metrically semantic" of fuzzy logics.

To adopt connectives suitably with the metrically semantic, firstly, the strong conjunction $\&$ would be interpreted by a continuous t-conorm instead of a continuous t-norm. A continuous t-conorm $S$ is a continuous function $S:[0,1]^2\to [0,1]$
(in Euclidean topology) commutative, associative, non-decreasing on both arguments, in which $S(0,x)=x$ for all $x\in[0,1]$. One could easily derived that $S(1,x)=1$ for all $x\in[0,1]$.

The appropriate interpretation of the implication connective $\dotto:[0,1]^2\to[0,1]$ is the residuum of the t-conorm $S$, defined by the adjoint property,
\begin{center}
for all $x, y, z\in[0,1]$, $S(z,x)\ge y$ iff $z\ge x\dotto y$.
\end{center}
The continuity of $S$ implies that $x\dotto y=\min\{z: S(z,x)\ge y\}$. The
well known continuous t-conorms and their residua are listed in Table~\ref{tconorms}.
\begin{table}[h]
\centering
\caption{Some continuous t-conorms and their residua}
\label{tconorms}
\begin{tabular}{l l l l}
\toprule
\rowcolor[gray]{.8}
logic name&t-conorm&residuum\\
\midrule
\lo&$S_L(x,y)=\min\{x+y,1\}$&$x\dotto_L y=
\left\{\begin{array}{cc}
0& x\ge y\\
y-x& x<y
\end{array}\right.$\\
\g&$S_G(x,y)=\max\{x,y\}$&$x\dotto_G y=
\left\{\begin{array}{cc}
0& x\ge y\\
y& x<y
\end{array}\right.$\\
Product&$S_\pi(x,y)=x+y-xy$&$x\dotto_\pi y=
\left\{\begin{array}{cc}
0& x\ge y\\
\displaystyle\frac{y-x}{1-x}&x<y
\end{array}\right.$\\
\bottomrule
\end{tabular}
\end{table}

In metrically semantic, an evaluation is a map $v$ from the set of all propositions to $[0,1]$ with the following properties
\begin{itemize}
\item $v(\bot)=1$,
\item $v(\varphi \&\psi)=S\left(v(\varphi),v(\psi)\right)$,
\item $v(\varphi\to \psi)=v(\varphi)\dotto v(\psi)$.
\end{itemize}
$v$ models $\varphi$ whenever $v(\varphi)=0$. Other concepts are defined in a similar way. For other logical connectives, interpretations in metrically semantic could be calculated relevantly. For example, since $S$ is continuous, one could easily verify that
\begin{eqnarray*}
&&v(\varphi\wedge \psi)=
\max\{v(\varphi),v(\psi)\},\\
&&v(\varphi\vee \psi)=
\min\{v(\varphi),v(\psi)\},
\end{eqnarray*}
which are the dual of their interpretations in the semantic based on continuous t-norms.

In the predicate case, for a first-order language $\mathcal{L}$ and an $\mathcal{L}$-structure $\mathcal{M}$, we could dedicated the following interpretations in metrically semantic.
\begin{itemize}
\item If $\varphi(\bar{x})=\forall y\ \psi(y,\bar{x})$ then $\varphi^\mathcal{M}(\bar{a})=\displaystyle\sup_{b\in M}\{\psi^\mathcal{M}(b,\bar{a})\}$
\item If $\varphi(\bar{x})=\exists y\ \psi(y,\bar{x})$ then $\varphi^\mathcal{M}(\bar{a})=\displaystyle\inf_{b\in M}\{\psi^\mathcal{M}(b,\bar{a})\}$
\end{itemize}
For two t-conorms $S_1$ and $S_2$, $S_1$ is weaker than $S_2$, in symbols $S_1\le S_2$, whenever $S_1(x,y)\le S_2(x,y)$ for all $x,y\in[0,1]$. Obviously, $S_G\le S_\pi\le S_L$.

The axioms of {\bf BL} are hold here as well. However, note that their semantical meanings are as the dual ones in the everyday semantic. The following facts about arbitrary continuous t-conorm $S$ and it's residua $\dotto$, are used in the further arguments.
\begin{lemma}\label{usefullem}
For each continuous t-conorm $S$ and it's residua $\dotto$, the followings are true.
\begin{enumerate}[label={\thelemma.\arabic*}]
\item $S(x,x\dotto y)=\max\{x,y\}$ \label{fa1}
\item $x\dotto y\ge (y\dotto z)\dotto (x\dotto z)$ \label{fa2}
\item $y\ge x\dotto y$ \label{fa3}
\item $S(x,y)\ge\max\{x,y\}$ \label{fa4}
\item $S(x\dotto y,x'\dotto y')\ge S(x,x')\dotto S(y,y')$ \label{fa5}
\end{enumerate}
\end{lemma}
\begin{proof}
\ref{fa1} is follows from the definition of $\wedge$ that is $\varphi\wedge\psi:=\varphi\&(\varphi\to\psi)$.
\ref{fa2} is an obvious consequence of (A1). \ref{fa3}, \ref{fa4}, and \ref{fa5} are follows from
\ref{l1}, \ref{l2}, and \ref{l3}, respectively.
\end{proof}
The main idea that we chose the metrically semantic is the interpretation of the equivalence connective. Indeed an easy argument show that for any continuous t-conorm $S$ weaker than the \lo t-conorm, the interpretation of the equivalence connective is a metric on $[0,1]$.
\begin{theorem}\label{metric}
Let $S$ be a continuous t-conorm and $\dotto$ be the residue of $S$. Then, for any $x,y,z\in[0,1]$,
\begin{center}
$x\dotto y\le S\left(x\dotto z,z\dotto y\right)$.
\end{center}
Specially if $S$ is weaker than the \lo t-conorm, then the interpretation of the
equivalence connective is a metric on $[0,1]$.
\end{theorem}
\begin{proof}
Define $d:[0,1]^2\to[0,1]$ by $d(x,y)=\left\{\begin{array}{cc}
x\dotto y& x\le y\\
y\dotto x& x>y
\end{array}\right.$. As $x\dotto y=0$ for each $x\ge y$, and $S(x,0)=x$ for each $x$, one could easily verify that
for any evaluation $v$, $v(\varphi \leftrightarrow \psi)= d\left(v(\varphi),v(\psi)\right)$.

Obviously, $d(x,x)=0$ for $x\in[0,1]$. On the other hand, if $d(x,y)=0$ and we assume that $x<y$, then
 $d(x,y)=x\dotto y=0$. But $x\dotto y=\min\{z: S(z,x)\ge y\}$. Thus, $0\in\{z: S(z,x)\ge y\}$ that is $S(0,x)\ge y$ which means that $x\ge y$, a contradiction. By symmetry, $x>y$ also leads to a contradiction. So $x=y$.

Symmetric property of $d$ is clear. In order to prove the triangle inequality, by Lemma \ref{fa2}
for arbitrary $x,y,z,\in [0,1]$ we have $x\dotto z\ge (z\dotto y)\dotto (x\dotto y)$.
Now, adjointness of $S$ and $\dotto$ implies that
\begin{center}
$S\left(x\dotto z,z\dotto y\right)\ge x\dotto y$.
\end{center}
Furthermore, since $S\le S_L$ we have
\begin{center}
$x\dotto y\le S\left(x\dotto z,z\dotto y\right)
\le S_L\left(x\dotto z,z\dotto y\right)\le d(x,z)+d(z,y)$.
\end{center}
A similar argument show that $y\dotto x\le d(x,z)+d(z,y)$, that completes the proof of the triangle inequality.
\end{proof}

The corresponding metrics which interprets the equivalence connective of the logics listed in
Table \ref{tconorms} are proposed in Figure \ref{fig1}. Note that the white color in Figure \ref{fig1} is the absolute truth while the black color describe the absolute falsity.
\begin{figure}[h]
  \centering
  \begin{minipage}[b]{0.3\textwidth}
    \includegraphics[width=\textwidth]{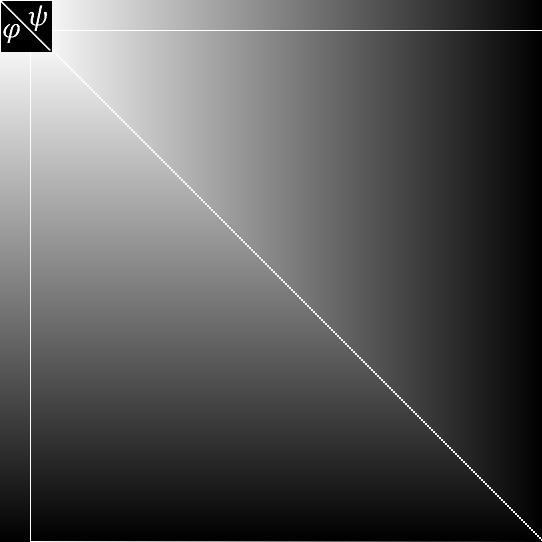}
  \centering\scalebox{0.9}[1]{$d_G(x,y)=\left\{\begin{array}{ll}
\max\{x,y\}&x\ne y\\
0&x=y
\end{array}\right.$}
  \end{minipage}
  \hfill
  \begin{minipage}[b]{0.3\textwidth}
    \includegraphics[width=\textwidth]{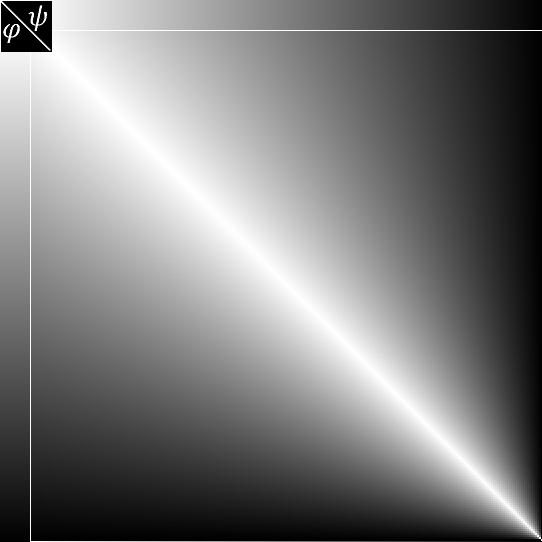}
  \centering\scalebox{0.9}[1]{
  \medmuskip=-3mu
\thinmuskip=-2mu
\thickmuskip=-2mu
\nulldelimiterspace=-1pt
\scriptspace=0pt
  $d_\pi(x,y)=\left\{\begin{array}{@{}l@{}l}
|x-y|/(1-\min\{x,y\})&x\ne y\\
0&x=y
\end{array}\right.$}
  \end{minipage}
  \hfill
  \begin{minipage}[b]{0.3\textwidth}
    \includegraphics[width=\textwidth]{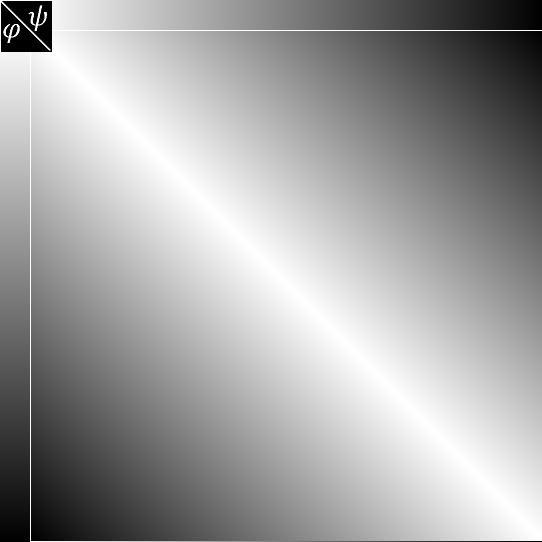}
  \centering\scalebox{0.9}[1]{$\begin{array}{ll}
d_L(x,y)=|x-y|&\\
&\end{array}$}
  \end{minipage}
\caption{interpretation of the equivalence connective in metrically semantic}
\label{fig1}
\end{figure}

The metric $d$ introduced in Theorem \ref{metric}, induced a metric $\mathbf{d}$ on $[0,1]^2$ as follows.
\begin{lemma}\label{metricd}
Let $S$ be a continuous t-conorm weaker than $S_L$ and $\dotto$ be it's residua. Furthermore let $d$ be the metric
defined in Theorem \ref{metric}. The mapping $\mathbf{d}\big((x_1,x_2),(y_1,y_2)\big)=S\big(d(x_1,y_1),d(x_2,y_2)\big)$
define a metric on $[0,1]^2$.
\end{lemma}
\begin{proof}
Let's denote $(x_1,x_2)$ by $\bar{x}$. We use this notation hereafter.
Obviously, $\mathbf{d}(\bar{x},\bar{y})=0$ if and only if $\bar{x}=\bar{y}$. Furthermore, using the symmetric
property of $d$ we get it for $\mathbf{d}$. For transitivity let $\bar{x},\bar{y},\bar{z}\in [0,1]^2$. Using Remark \ref{dlesd} and associativity of t-conorm $S$ the proof will be completed.
\begin{eqnarray*}
  \mathbf{d}(\bar{x},\bar{y}) &=& S\big(d(x_1,y_1),d(x_2,y_2)\big)\\
  &\le& S\Big(S\big(d(x_1,z_1),d(z_1,y_1)\big),S\big(d(x_2,z_2),d(z_2,y_2)\big)\Big)\\
  &=& S\Big(S\big(d(x_1,z_1),d(x_2,z_2)\big),S\big(d(z_1,y_1),d(z_2,y_2)\big)\Big)\\
  &=&S\Big(\mathbf{d}(\bar{x},\bar{z}),\mathbf{d}(\bar{z},\bar{y})\Big)\\
  &\le&S_L\Big(\mathbf{d}(\bar{x},\bar{z}),\mathbf{d}(\bar{z},\bar{y})\Big)\\
  &\le&\mathbf{d}(\bar{x},\bar{z})+\mathbf{d}(\bar{z},\bar{y})
\end{eqnarray*}
\end{proof}
The following theorem show that why we could use the metric $d$ to prove the compactness theorem. Verily, the interpretation of all logical connectives in metrically semantic are continuous functions with respect to the topology induced by metric $d$ on $[0,1]$ and $[0,1]^2$.
\begin{theorem}\label{con}
Assume that $S$, $\dotto$, $d$ and $\mathbf{d}$ be as in the Lemma \ref{metricd}. Then
$S:([0,1]^2,\mathbf{d})\to([0,1],d)$ and $\dotto:([0,1]^2,\mathbf{d})\to([0,1],d)$ are continuous functions.
\end{theorem}
\begin{proof}
Let $\bar{x},\bar{y}\in[0,1]^2$ and assume that $S(x_1,x_2)\ge S(y_1,y_2)$. By using Lemma \ref{fa5} we have
\begin{eqnarray*}
  d\big(S(x_1,x_2), S(y_1,y_2)\big)&=&S(x_1,x_2)\dotto S(y_1,y_2)\\
  &\le& S(x_1\dotto y_1, x_2\dotto y_2)\\
  &\le& S\big(d(x_1,y_1),d(x_2,y_2)\big)\\
  &=&\mathbf{d}(\bar{x},\bar{y}).
\end{eqnarray*}
If $S(x_1,x_2)< S(y_1,y_2)$ a similar argument show that $d\big(S(x_1,x_2), S(y_1,y_2)\big)\le\mathbf{d}(\bar{x},\bar{y})$. Thus $S$ is a uniformly continuous function. For continuity of $\dotto$ by Lemma \ref{fa2}
\begin{center}
$x_1\dotto y_1\ge (y_1\dotto y_2)\dotto (x_1\dotto y_2)$,
\end{center}
which is alongside the adjointness of $S$ and $\dotto$ implies that
\begin{equation}\label{z1}
S\big((x_1\dotto y_1),(y_1\dotto y_2)\big)\ge x_1\dotto y_2.
\end{equation}
Again using Lemma \ref{fa2} we get $x_1\dotto y_2\ge (y_2\dotto x_2)\dotto (x_1\dotto x_2)$. Now, Beside inequality \ref{z1} we have
\begin{center}
$S\big((x_1\dotto y_1),(y_1\dotto y_2)\big)\ge (y_2\dotto x_2)\dotto (x_1\dotto x_2).$
\end{center}
Once more, since $\dotto$ is the residua of $S$ we have
\begin{center}
$S\Big(S\big((x_1\dotto y_1),(y_1\dotto y_2)\big),(y_2\dotto x_2)\Big)\ge(x_1\dotto x_2).$
\end{center}
Now, due to the commutativity and associativity of $S$ we get
\begin{center}
$S\Big(S\big((x_1\dotto y_1),(y_2\dotto x_2)\big),(y_1\dotto y_2)\Big)\ge(x_1\dotto x_2).$
\end{center}
Once again, adjointness of $S$ and $\dotto$ gives
\begin{eqnarray*}
  (y_1\dotto y_2)\dotto(x_1\dotto x_2) &\le&S\big((x_1\dotto y_1),(y_2\dotto x_2)\big)\\
  &\le&S\big(d(x_1,y_1),d(x_2,y_2)\big)\\
  &=&\mathbf{d}(\bar{x},\bar{y}),
\end{eqnarray*}
A similar argument show that
$(x_1\dotto x_2)\dotto(y_1\dotto y_2)\le\mathbf{d}(\bar{x},\bar{y})$. Whence
\begin{center}
$d\big((x_1\dotto x_2),(y_1\dotto y_2)\big)\le\mathbf{d}(\bar{x},\bar{y}),$
\end{center}
which completes the proof.
\end{proof}
\section{Compactness and $K$-Compactness in Basic Logic: New Results}
In this section using the continuity of logical connectives with respect to the metric introduced in Lemma
\ref{metricd}, we prove some versions of the compactness for fuzzy logics. In the rest of this section, whenever we deal with satisfiability, we mean satisfiability in metric semantic.
\subsection{propositional basic logic}\hfill\\
In the propositional case, the compactness in general could be prove as in the propositional \lo logic \cite{but1995,cintnav2004}.
\begin{theorem}\label{kcompfuzzy}
Let $S$, $\dotto$, and $d$ be as in the Lemma \ref{metricd} and furthermore assume that
$K$ be a compact subset of $([0,1],d)$. Then in metric semantic, every finitely $K$-satisfiable theory over BL is $K$-satisfiable.
\end{theorem}
\begin{proof}
Let $P$ and $Prop$ be as in the Definition \ref{bllogic}. Since every assignment $v_0:P\to [0,1]$ determine a unique evaluation $v:Prop \to [0,1]$, So $[0,1]^I$ determine the set of all evaluations.
Now as by Theorem \ref{con}, logical connectives are interpreted by continuous functions,
each $\varphi\in Prop$ can be identified by a continuous function $\hat{\varphi}:[0,1]^I\to[0,1]$ defined by
$\hat{\varphi}(v)=v(\varphi)$.

Now, assume that $\Sigma$ be a finitely $K$-satisfiable theory. Thus, for each finite subset $\Sigma_0$ of $\Sigma$, $\bigcap_{\varphi\in\Sigma_0}\hat{\varphi}^{-1}(K)\ne\emptyset$. But, for each $\varphi\in\Sigma$, $\hat{\varphi}^{-1}(K)$ is a compact subset of $[0,1]^I$ and so is $\bigcap_{\varphi\in\Sigma_0}\hat{\varphi}^{-1}(K)$. Now,
finite intersection property of compact sets implies that, $\bigcap_{\varphi\in\Sigma}\hat{\varphi}^{-1}(K)\ne\emptyset$, that is $\Sigma$ is $K$-satisfiable.
\end{proof}

By Theorem \ref{nonkcomplo} in the case of \lo logic as well as it's dual, for any noncompact subset $K$ of $([0,1],d_L)$, $K$-compactness fails in \lo logic. However, for arbitrarily continuous t-conorm based fuzzy logics, this is not hold. Indeed the expressive power of the language of logic, imposes some limitations in the results.

In the case of \g logic and product logic, this limitation is stated in Fact \ref{kcompgodelprop}. The translation of this Fact in metric semantic is as follows.
\begin{theorem}\label{01comgodel}
Assume that $K\subseteq[0,1)$ containing $0$. In metric semantic, {\bf G} as well as $\mathbf{\Pi}$ is $K$-compact.
\end{theorem}
By Theorem \ref{01comgodel} for example if we set $K=[0,1)$, then {\bf G} is $K$-compact but $K$ is not a compact subset of $([0,1],d_G)$. However, in Example \ref{counterexamplegodelp} we state a weak version of necessity condition for the $K$-compactness of {\bf G}. The translation of this example to the metric semantic is as follows.
\begin{example}\label{counterexamplegodelpmetric}
Assume that $K$ be an infinite subset of $(0,1]$ and $T=\{(p_i\to p_j\}_{i\le j, i,j\in \omega_2}$. Then in metric semantic, $T$ is finitely $K$-satisfiable but it is not $K$-satisfiable.
\end{example}
Note that infinite subset of $(0,1]$ are not compact in $([0,1],d_G)$. Indeed the only compact subsets of $([0,1],d_G)$ are finite subsets or countably infinite subsets which contains $0$ as the only limit point with respect to the order topology. However, for $K=[0,1]$ we have neither a proof nor a counter-example for $K$-compactness of propositional \g logic with respect to arbitrary set of atomic propositions. Hence, the following corollary summarizes the results of this section for propositional \g logic.
\begin{corollary}
Assume that $K\subsetneqq[0,1]$. In metric semantic, the propositional \g logic admit $K$-compactness if and only if $K$ is either a compact subset of $([0,1],d_G)$ or contains $0$ but does not contain $1$.
\end{corollary}
In the case of product logic, one could easily verify that the open balls with center $a$ and radius $r$ in $([0,1],d_\pi)$ are as follows:
\begin{itemize}
  \item[] $r<a<1$ : $N_r(a)=((a-r)/(1-r),a+r-ar)$,
  \item[] $a<r\le 1$ : $N_r(a)=[0,a+r-ar)$,
  \item[] $a=1$ : $N_r(a)=\{1\}$.
\end{itemize}
Now, an easy argument shows $K=[0,1)$ is not a compact subset of $([0,1],d_\pi)$, but Theorem \ref{01comgodel} says that the product logic is $K$-compact.

On the other hand, there is no characterization for product logic propositions like as the  McNaughton's characterization for \lo propositions. So, we could not state a suitable condition for the necessity condition of the $K$-compactness in product logic like as the one in \lo lgoic (Fact \ref{nonkcomplo}).
\subsection{first-order basic logic}\hfill\\
There are several approaches to prove the compactness of the firs-order \lo logic. \cite{chang63} apply the concept of proof and consistency and then using the continuity of interpretation of logical connectives, show that consistency and satisfiability are equivalent concepts. \cite{pav79} add some nullary connectives and again using the continuity of interpretation of logical connectives show that truth degree of any sentence is equal to its provability degree. \cite{tpd2011} use the ultraproduct method which again used the continuity of interpretation of logical connectives.
We use the ultraproduct method to proof a

sufficiency condition for the $K$-compactness of {\L}$\forall$ for arbitrary compact subset $K$ of $[0,1]$, and it is proved in \cite{tpd2011}.
To use the ultraproduct method, lets remind some facts about filters on topological spaces.
\begin{enumerate}[label=\textbf{(Fact\arabic*)},ref=(Fact\arabic*),leftmargin=1.5cm]
\item\label{f1} A filter $\mathfrak{D}$ on a topological space $X$, convergent to
an element $x\in X$, whenever for each open set $U$ containing $x$, $U$ is an element of $\mathfrak{D}$.
This is denoted by $\mathfrak{D}\to x$ and $x$ is called a limit point of $\mathfrak{D}$.
\item\label{f2} $X$ is a compact Hausdorff space if and only if every filter
$\mathcal{D}$ on $X$ has a unique limit point.
\item\label{f3} Let $f:X\to Y$ be a continuous function at $x_0\in X$ and $\mathfrak{D}$ be a filter on $X$.
If $f(\mathfrak{D})$ be the filter on $Y$ generated by the set $\{f(A): A\in\mathfrak{D}\}$,
then $f(\mathfrak{D})\to f(x_0)$.
\end{enumerate}
\begin{defn}
Let $X$ be a topological space, $I$ be a nonempty set, and $\mathfrak{D}$ be a
filter on $I$. Furthermore, let $f\in I^X$,  $\{x_i\}_{i\in I}$ be the range of $f$, and
$f^*(\mathfrak{D})=\{A\subseteq X: f^{-1}(A)\in\mathfrak{D}\}$. If $f^*(\mathfrak{D})$
is convergent to $x\in X$, then we call
$x$ the $\mathfrak{D}$-limit of the family $\{x_i\}_{i\in I}$ and write
$\lim_\mathfrak{D}x_i=x$.
\end{defn}
Another version of \ref{f3}, is the following.
\begin{corollary}\label{limit}
Let $f:X\to Y$ be a continuous function at $x_0\in X$, $I$ be a
nonempty set, and $\mathfrak{D}$ be a filter on $I$. If $\lim_\mathfrak{D}x_i=x_0$
then $\lim_\mathfrak{D}f(x_i)=f(x_0)$.
\end{corollary}
\begin{proof}
Assume that $\{x_i\}_{i\in I}$ be the range of the function $\alpha\in X^I$. So,
$\alpha^*(\mathfrak{D})\to x_0$ and by \ref{f3}
$f\big(\alpha^*(\mathfrak{D})\big)\to f(x_0)$. Now, if we show that
$f\big(\alpha^*(\mathfrak{D})\big)\subseteq(\alpha \circ f)^*(\mathfrak{D})$, then
we have $(\alpha \circ f)^*(\mathfrak{D})\to f(x_0)$ which fulfills the proof.

Let $B\in f\big(\alpha^*(\mathfrak{D})\big)$. So, there
exists $A\in\alpha^*(\mathfrak{D})$ such that $f(A)\subseteq B$. Hence,
$A\subseteq f^{-1}(B)$ and therefore $\alpha^{-1}(A)\subseteq \alpha^{-1}\big(f^{-1}(B)\big)$.
But $A\in\alpha^*(\mathfrak{D})$ and so $\alpha^{-1}(A)\in\mathfrak{D}$ which implies
that $\alpha^{-1}\big(f^{-1}(B)\big)\in\mathfrak{D}$.
Whence, $B\in (\alpha \circ f)^*(\mathfrak{D})$.
\end{proof}
\begin{lemma}\label{order}
Let $V$ be a \g set, $I$ be a nonempty set, and $\mathfrak{D}$ be a filter on $I$.
Consider $(V,d_{max})$ as a topological space.
If $\{x_i\}_{i\in I}$ and $\{y_i\}_{i\in I}$ are two family of elements of $V$, then
\begin{center}
$\lim_\mathfrak{D}x_i\le\lim_\mathfrak{D}y_i$ if and only if $\{i :x_i\le y_i\}\in\mathfrak{D}$.
\end{center}
\end{lemma}
\begin{proof}
Let $E=\{i:x_i\le y_i\}$. Assume that for each $i\in E$, $x_i\le y_i$. Thus,
for each $i\in E$, $y_i\dotto x_i=0$. Now, by continuity of $\dotto:(V^2,d_{Max})\to(V,d_{max})$
and using the Corollary \ref{limit} we get $y\dotto x=0$. Thus, $x\le y$.

Conversely, if $x\le y$, then $\lim_\mathfrak{D}y_i\dotto x_i=0$. So,
$\{i: x_i\le y_i\}\in\mathfrak{D}$.
\end{proof}
Now, assume that $V$ is a \g set which is a
compact Hausdorff subspace of $([0,1],d_{max})$. For example
assume that $V=V_\downarrow=\{\frac{1}{n}: n\in\mathbb{N}\}\cup\{0\}$.
Then by \ref{f2}, we could construct the ultraproduct of
a family of structures in the first-order \g logic $\mathfrak{G}_V$.
\begin{defn}
Let $\{\mathcal{M}_i\}_{i\in I}$ be a family of $\mathcal{L}$-structures and
$\mathfrak{D}$ be a filter on $I$. The $\mathfrak{D}$-ultraproduct of family
$\{\mathcal{M}_i\}_{i\in I}$ is an $\mathcal{L}$-structure $\mathcal{M}$ with
universe $M=\prod_{i\in I}M_i$ whose interpretation of elements
of $\mathcal{L}$ is defined as follows.
\begin{itemize}
\item For $n$-ary predicate symbol $R\in\mathcal{L}$, $R^\mathcal{M}:M^n\to V$ is
defined by
\begin{center}
$R^\mathcal{M}\big(\{x_i^1\}_{i\in I}, ..., \{x_i^n\}_{i\in I}\big)
=\lim_\mathfrak{D}R^{\mathcal{M}_i}(x_i^1, ..., x_i^n)$.
\end{center}
\item For $n$-ary function symbol $f\in\mathcal{L}$, $f^\mathcal{M}:M^n\to M$ is
defined by
\begin{center}
$f^\mathcal{M}\big(\{x_i^1\}_{i\in I}, ..., \{x_i^n\}_{i\in I}\big)
=\{f^{\mathcal{M}_i}(x_i^1, ..., x_i^n)\}_{i\in I}$.
\end{center}
\end{itemize}
\end{defn}
Obviously, by \ref{f2} the above definition is well-defined.
\begin{theorem}({\L}o$\acute{s}$ theorem)~Let $V$ be a \g set and $(V,d_{max})$ be a compact Hausdorff space.
Furthermore, assume that $\{\mathcal{M}_i\}_{i\in I}$ be a family of $\mathcal{L}$-structures.
If $\mathfrak{D}$ is an ultrafilter on $I$ and $\mathcal{M}$ is the
$\mathfrak{D}$-ultraproduct of family $\{\mathcal{M}_i\}_{i\in I}$, then in first-order \g logic $\mathfrak{G}_V$,
for each $\mathcal{L}$-formula $\varphi(x_1, ..., x_n)$ and each
$\mathbf{a}_k=\{a_i^k\}_{i\in I}\in M$ $(1\le k\le n)$,
\begin{center}
$\displaystyle\varphi^\mathcal{M}(\mathbf{a}_1, ..., \mathbf{a}_n)=\lim_\mathfrak{D}\varphi^{\mathcal{M}_i}(a_i^1, ..., a_i^n).$
\end{center}
\end{theorem}
\begin{proof}
The proof is by induction on formulas.
\begin{itemize}
\item Clearly, for every atomic formula, by definition of the $\mathfrak{D}$-ultraproduct of family
$\{\mathcal{M}_i\}_{i\in I}$,
\begin{center}
$R^\mathcal{M}\big(\mathbf{a}_1, ..., \mathbf{a}_n\big)
=\lim_\mathfrak{D}R^{\mathcal{M}_i}(a_i^1, ..., a_i^n).$
\end{center}
\item Let $\varphi(\bar{x})=\theta(\bar{x})\to\psi(\bar{x})$, where for each
$\mathbf{a}_k=\{a_i^k\}_{i\in I}\in M$ $(1\le k\le n)$,
\begin{center}
$\displaystyle\theta^\mathcal{M}(\mathbf{a}_1, ..., \mathbf{a}_n)=\lim_\mathfrak{D}\theta^{\mathcal{M}_i}(a_i^1, ..., a_i^n),~~~~
\psi^\mathcal{M}(\mathbf{a}_1, ..., \mathbf{a}_n)=\lim_\mathfrak{D}\psi^{\mathcal{M}_i}(a_i^1, ..., a_i^n).$
\end{center}
Assume that
\begin{center}
$V_0=\{\chi^\mathcal{M}(\bar{\mathbf{a}}):~\chi(\bar{x})~\text{is an}~\mathcal{L}\text{-formula and}~
\bar{\mathbf{a}}\subseteq M\}\cup
\{\chi^{\mathcal{M}_i}(\bar{a}):~\chi(\bar{x})~\text{is an}~\mathcal{L}\text{-formula and}~
\bar{a}\subseteq M_i\}_{i\in I}$.
\end{center}
As $\mathfrak{D}$ is an ultrafilter on $I$, since $\dotto$ is a continuous function by
Corollary \ref{limit},
\begin{eqnarray*}
\displaystyle\varphi^\mathcal{M}(\mathbf{a}_1, ..., \mathbf{a}_n)&=&
\theta^\mathcal{M}(\mathbf{a}_1, ..., \mathbf{a}_n)\dotto\psi^\mathcal{M}(\mathbf{a}_1, ..., \mathbf{a}_n)\\
&=&\lim_\mathfrak{D}\theta^{\mathcal{M}_i}(a_i^1, ..., a_i^n)\dotto
\lim_\mathfrak{D}\psi^{\mathcal{M}_i}(a_i^1, ..., a_i^n)\\
&=&\lim_\mathfrak{D}\big(\theta^{\mathcal{M}_i}(a_i^1, ..., a_i^n)\dotto\psi^{\mathcal{M}_i}(a_i^1, ..., a_i^n)\big)\\
&=&\lim_\mathfrak{D}\varphi^{\mathcal{M}_i}(a_i^1, ..., a_i^n)
\end{eqnarray*}
\item $\varphi(\bar{x})=\theta(\bar{x})\wedge\psi(\bar{x})$ is analogous to the previous item.
\item Let $\varphi(x_1, ..., x_n)=\forall y\,\psi(x_1, ..., x_n,y)$, where for each
$\mathbf{c}=\{c_i\}_{i\in I}\in M$ and
$\mathbf{a}_k=\{a_i^k\}_{i\in I}\in M$ $(1\le k\le n)$,
\begin{center}
$\psi^\mathcal{M}(\mathbf{a}_1, ..., \mathbf{a}_n,\mathbf{c})=\lim_\mathfrak{D}\psi^{\mathcal{M}_i}(a_i^1, ..., a_i^n,c_i).$
\end{center}
For each $i\in I$, $\mathbf{a}_1, ..., \mathbf{a}_n ,\mathbf{c}\in M$,
\begin{center}
$\psi^{\mathcal{M}_i}(a_i^1, ..., a_i^n,c_i)\le\sup_{c_i\in M_i}\psi^{\mathcal{M}_i}(a_i^1, ..., a_i^n,c_i)=\varphi^{\mathcal{M}_i}(a_i^1, ..., a_i^n)$.
\end{center}
Thus, by Lemma \ref{order},
\begin{center}
$\lim_\mathfrak{D}\psi^{\mathcal{M}_i}(a_i^1, ..., a_i^n,c_i)\le
\lim_\mathfrak{D}\varphi^{\mathcal{M}_i}(a_i^1, ..., a_i^n)$.
\end{center}
So,
\begin{eqnarray*}
\varphi^\mathcal{M}(\mathbf{a}_1, ..., \mathbf{a}_n)
&=&\sup_{\mathbf{c}\in M}\psi^\mathcal{M}(\mathbf{a}_1, ..., \mathbf{a}_n,\mathbf{c})\\
&=&\sup_{\mathbf{c}\in M}\lim_\mathfrak{D}\psi^{\mathcal{M}_i}(a_i^1, ..., a_i^n,c_i)\\
&\le&\lim_\mathfrak{D}\varphi^{\mathcal{M}_i}(a_i^1, ..., a_i^n).
\end{eqnarray*}
To prove the reverse inequality, we show that for each $v\in V$,
\begin{center}
if $\varphi^\mathcal{M}(\mathbf{a}_1, ..., \mathbf{a}_n)<v$ then
$\lim_\mathfrak{D}\varphi^{\mathcal{M}_i}(a_i^1, ..., a_i^n)<v$.
\end{center}
Suppose for the propose of contradiction
that $\varphi^\mathcal{M}(\mathbf{a}_1, ..., \mathbf{a}_n)<v$ but
$v\le\lim_\mathfrak{D}\varphi^{\mathcal{M}_i}(a_i^1, ..., a_i^n)$. Thus,
\begin{center}
$E=\{i: v\le\varphi^{\mathcal{M}_i}(a_i^1, ..., a_i^n)\}\in\mathfrak{D}$.
\end{center}
So, for each $i\in E$, $v\le\sup_{c_i\in M_i}\psi^{\mathcal{M}_i}(a_i^1, ..., a_i^n,c_i)$, which means
that for each $i\in E$ there is $b_i\in M_i$ such
that $v\le\psi^{\mathcal{M}_i}(a_i^1, ..., a_i^n,b_i)$. Consider some arbitrary $b_i\in M_i$ for $i\notin E$ and
let $\mathbf{b}=\{b_i\}_{i\in I}$. By Lemma \ref{order}, we have
\begin{eqnarray*}
v&\le&\lim_\mathfrak{D}\psi^{\mathcal{M}_i}(a_i^1, ..., a_i^n,b_i)\\
&=&\psi^\mathcal{M}(\mathbf{a}_1, ..., \mathbf{a}_n,\mathbf{b})\\
&\le&\sup_{\mathbf{c}\in M}\psi^\mathcal{M}(\mathbf{a}_1, ..., \mathbf{a}_n,\mathbf{c})\\
&=&\varphi^\mathcal{M}(\mathbf{a}_1, ..., \mathbf{a}_n),
\end{eqnarray*}
a contradiction.
\item $\varphi(x_1, ..., x_n)=\exists y\,\psi(x_1, ..., x_n,y)$, is similar to the previous item.
\end{itemize}
\end{proof}
\begin{theorem}(Compactness theorem)~Let $V$ be a \g set and $(V,d_{max})$ be a compact Hausdorff space.
In first-order \g logic $\mathfrak{G}_V$, every finitely satisfiable theory is satisfiable.
\end{theorem}
\begin{proof}
Assume that $T$ is a finitely satisfiable theory. Let $I$ be the set of all finite subsets of $T$.
For each $\varphi\in T$, let $\overline{\varphi}=\{\Sigma: \varphi\in\Sigma~\text{and}~\Sigma\in I\}$.
Obviously $\mathfrak{T}=\{\overline{\varphi}: \varphi\in T\}$ has the finite intersection property. So,
there exists an ultrafilter $\mathfrak{D}$ on $I$ containing $\mathfrak{T}$.

Let $T_i\in I$. As $T$ is finitely satisfiable, there exists a structure $\mathcal{M}_i\models T_i$.
Suppose that $\mathcal{M}$ be the $\mathfrak{D}$-ultraproduct of $\{\mathcal{M}_i\}_{i\in I}$.
By {\L}o$\acute{s}$ theorem, $\mathcal{M}\models T$.
\end{proof}

%
\bibliographystyle{mlq}
\bibliography{bibilio}
%
\end{document}